\documentclass{elsarticle}
\usepackage{amsfonts}

\usepackage{graphicx}
\usepackage{pstricks}
\usepackage{amsmath}
\usepackage{amsxtra}

\newtheorem{theorem}{Theorem}[section]
\newtheorem{lemma}[theorem]{Lemma}
\newproof{proof}{Proof}
\newtheorem{proposition}[theorem]{Proposition}
\newtheorem{corollary}[theorem]{Corollary}

\newtheorem{example}[theorem]{Example}

\numberwithin{equation}{section}

\begin{document}

\begin{frontmatter}

\title{Subnormality for arbitrary powers of 2-variable \\
weighted shifts whose restrictions to a \\ 
large invariant subspace are tensor products}

\author{Ra\'{u}l E. Curto}
\address{Department of Mathematics, The University of Iowa, Iowa City, Iowa
52242}
\ead{raul-curto@uiowa.edu}
\ead[url]{http://www.math.uiowa.edu/\symbol{126}rcurto/}

\author{Sang Hoon Lee}
\address{Department of Mathematics, Chungnam National University, Daejeon,
305-764, Republic of Korea}
\ead{slee@cnu.ac.kr}
\author{Jasang Yoon}
\address{Department of Mathematics, The University of Texas-Pan American,
Edinburg, Texas 78539}
\ead{yoonj@utpa.edu}
\ead[url]{http://www.math.utpa.edu/\symbol{126}yoonj/}

\begin{abstract}
The Lifting Problem for Commuting Subnormals (LPCS) asks for necessary and
sufficient conditions for a pair of subnormal operators on Hilbert space to
admit commuting normal extensions. \ We study LPCS within the class of 
commuting $2$-variable weighted shifts $\mathbf{T} \equiv (T_1,T_2)$ with subnormal components
$T_1$ and $T_2$, acting on the Hilbert space $\ell ^2(\mathbb{Z}^2_+)$ with canonical orthonormal basis 
$\{e_{(k_1,k_2)}\}_{k_1,k_2 \geq 0}$ . \ The \textit{core} of
a commuting $2$-variable weighted shift $\mathbf{T}$, $c(\mathbf{T})$, is the restriction of $\mathbf{T}$ to the 
invariant subspace generated by all vectors $e_{(k_1,k_2)}$ with $k_1,k_2 \geq 1$; we say that 
$c(\mathbf{T})$ is of \textit{tensor form} if it is unitarily equivalent to a shift 
of the form $(I \otimes W_\alpha, W_\beta \otimes I)$, 
where $W_\alpha$ and $W_\beta$ are subnormal unilateral weighted shifts. \ Given a $2$-variable weighted 
shift $\mathbf{T}$ whose core is of tensor form, we prove
that LPCS is solvable for $\mathbf{T}$ if and only if
LPCS is solvable for any power $\mathbf{T}^{(m,n)}:=(T^m_1,T^n_2)$ ($m,n\geq 1$). \ 
\end{abstract}

\begin{keyword}
jointly hyponormal pairs, subnormal pairs, $2$-variable weighted
shifts, tensor form, core

\medskip

\textit{2000 Mathematics Subject Classification.} \ Primary 47B20, 
47B37, 47A13, 28A50; Secondary 44A60, 47-04, 47A20

\medskip

The first named author was partially supported by NSF Grants
DMS-0400741 and DMS-0801168. \ The second named author was 
partially supported by a National Research Foundation of 
Korea Grant funded by the Korean Government (2010-0028171). \ 
The third named author was partially supported by a Faculty Research
Council Grant at The University of Texas-Pan American.

\end{keyword}

\end{frontmatter}



\section{\label{Pre}Introduction}

The Lifting Problem for Commuting Subnormals (LPCS) asks for necessary and
sufficient conditions for a pair of subnormal operators on Hilbert space to
admit commuting normal extensions. \ In previous work (\cite{CLY1}, \cite%
{CLY2}, \cite{ROMP}, \cite{CLY4}, \cite{CuYo1}, \cite{CuYo2}, \cite{CuYo3}, %
\cite{Yo2}) we have studied LPCS from a number of different approaches. \
One such approach is to consider commuting pairs $\mathbf{T} \equiv (T_{1},T_{2})$
of subnormal operators and to ask to what extent the existence of liftings
for the powers $\mathbf{T}^{(m,n)}:=(T_{1}^{m},T_{2}^{n})\;(m,n\geq 1)$ can
guarantee a lifting for $\mathbf{T}$. \ For the class of $2$-variable
weighted shifts $\mathbf{T}$, it is often the case
that the powers of $\mathbf{T}$ are less complex than the initial pair; thus it becomes
especially significant to unravel how subnormality behaves under the action $%
(m,n)\mapsto \mathbf{T}^{(m,n)}\;\;(h,\ell \geq 1)$.

Within the class of $2$-variable weighted shifts, we consider the subclass $\mathcal{TC}$
consisting of pairs whose cores are of tensor form; that is

\begin{equation*}
\mathcal{TC}%
:=\{\mathbf{T} \in \mathfrak{H}_{0}:c(\mathbf{T}) \text{ is of tensor form}\}
\end{equation*}
This subclass has proved to be 
particularly attractive, since it is possible to
separate, within it, subnormality from $k$-hyponormality; thus, results
about LPCS for pairs in $\mathcal{TC}$ are especially useful. \ The class $\mathcal{TC}$ is small enough to allow for a simple
description of its pairs, yet large enough to be used as test ground for many significant hypotheses.

Before we proceed, we briefly pause to establish our
terminology. \ 
For $\alpha \equiv \{\alpha _{n}\}_{n=0}^{\infty }$ a bounded sequence of
positive real numbers (called \textit{weights}), let $W_{\alpha }:\ell ^{2}(%
\mathbb{Z}_{+})\rightarrow \ell ^{2}(\mathbb{Z}_{+})$ be the associated
unilateral weighted shift, defined by $W_{\alpha }e_{n}:=\alpha
_{n}e_{n+1}\;($all $n\geq 0)$, where $\{e_{n}\}_{n=0}^{\infty }$ is the
canonical orthonormal basis in $\ell ^{2}(\mathbb{Z}_{+}).$ \ 
Similarly, consider double-indexed positive bounded sequences $\alpha _{%
\mathbf{k}},\beta _{\mathbf{k}}\in \ell ^{\infty }(\mathbb{Z}_{+}^{2})$ , $%
\mathbf{k}\equiv (k_{1},k_{2})\in \mathbb{Z}_{+}^{2}$, and let $\{e_{\mathbf{k}}\}_{\mathbf{k} \in \mathbb{Z}^2_+}$ be the
canonical orthonormal basis in $\ell ^{2}(\mathbb{Z}^2_{+})$. \ We define the $2$%
-variable weighted shift $\mathbf{T}\equiv
(T_{1},T_{2})$ acting on $\ell ^{2}(\mathbb{Z}^2_{+})$ by 
\begin{equation*}
T_{1}e_{\mathbf{k}}:=\alpha _{\mathbf{k}}e_{\mathbf{k+}\varepsilon
_{1}}\quad \text{and}\quad T_{2}e_{\mathbf{k}}:=\beta _{\mathbf{k}}e_{%
\mathbf{k+}\varepsilon _{2}},
\end{equation*}%
where $\mathbf{\varepsilon }_{1}:=(1,0)$ and $\mathbf{\varepsilon }%
_{2}:=(0,1)$. \ The \textit{core} of
a commuting $2$-variable weighted shift $\mathbf{T}$ (in symbols, $c(\mathbf{T})$) is the restriction 
of $\mathbf{T}$ to the 
invariant subspace generated by all vectors $e_{(k_1,k_2)}$ with $k_1,k_2 \geq 1$; we say that 
$c(\mathbf{T})$ is of \textit{tensor form} if it is unitarily equivalent to a shift 
of the form $(I \otimes W_\alpha, W_\beta \otimes I)$, 
where $W_\alpha$ and $W_\beta$ are subnormal unilateral weighted shifts. \ Figure \ref{FigureROMP} shows both
the weight and Berger measure diagrams of a typical pair in $\mathcal{TC}$. \ As shown in 
\cite{ROMP}, each $\mathbf{T} \in \mathcal{TC}$ is completely determined by five parameters, i.e., the
$1$-variable measures $\sigma$, $\tau$, $\xi$ and $\eta$, and the positive number $a \equiv \alpha_{(0,1)}$. \
As we mentioned before, $\mathcal{TC}$ is of substantial interest to us, since it provides a fertile ground to 
test results on subnormality and $k$-hyponormality, and in particular about the solubility of LPCS.

\setlength{\unitlength}{.85mm} \psset{unit=.85mm} 
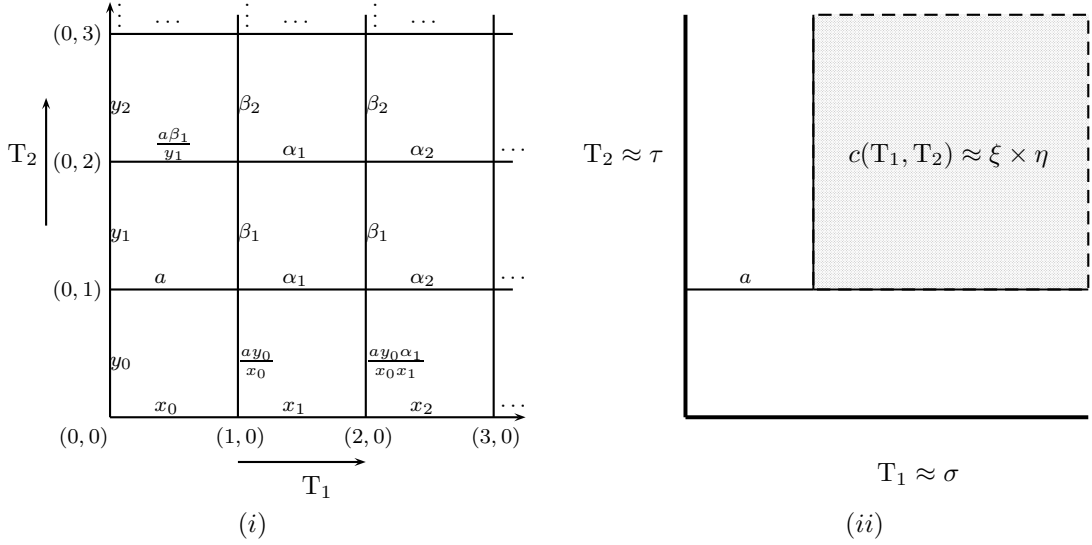
\begin{figure}[th]
\begin{center}
\begin{picture}(165,85)

\psline{->}(0,20)(65,20) 
\psline(0,40)(63,40)
\psline(0,60)(63,60) 
\psline(0,80)(63,80)
\psline{->}(0,20)(0,85) 
\psline(20,20)(20,83)
\psline(40,20)(40,83) 
\psline(60,20)(60,83)

\put(-8,16){\footnotesize{$(0,0)$}}
\put(16.5,16){\footnotesize{$(1,0)$}}
\put(36.5,16){\footnotesize{$(2,0)$}}
\put(56.5,16){\footnotesize{$(3,0)$}}

\put(7,21){\footnotesize{$x_{0}$}}
\put(27,21){\footnotesize{$x_{1}$}}
\put(47,21){\footnotesize{$x_{2}$}}
\put(61,21){\footnotesize{$\cdots$}}

\put(7,41){\footnotesize{$a$}}
\put(27,41){\footnotesize{$\alpha_{1}$}}
\put(47,41){\footnotesize{$\alpha_{2}$}}
\put(61,41){\footnotesize{$\cdots$}}

\put(7,62){\footnotesize{$\frac{a\beta_{1}}{y_{1}}$}}
\put(27,61){\footnotesize{$\alpha_{1}$}}
\put(47,61){\footnotesize{$\alpha_{2}$}}
\put(61,61){\footnotesize{$\cdots$}}

\put(7,81){\footnotesize{$\cdots$}}
\put(27,81){\footnotesize{$\cdots$}}
\put(47,81){\footnotesize{$\cdots$}}

\psline{->}(20,13)(40,13) \put(30,8){$\rm{T}_1$} 
\psline{->}(-10,50)(-10,70) \put(-16,60){$\rm{T}_2$}

\put(-9,39){\footnotesize{$(0,1)$}}
\put(-9,59){\footnotesize{$(0,2)$}}
\put(-9,79){\footnotesize{$(0,3)$}}

\put(0,28){\footnotesize{$y_{0}$}}
\put(0,48){\footnotesize{$y_{1}$}}
\put(0,68){\footnotesize{$y_{2}$}}
\put(1,81){\footnotesize{$\vdots$}}

\put(20,28){\footnotesize{$\frac{ay_{0}}{x_{0}}$}}
\put(20,48){\footnotesize{$\beta_{1}$}}
\put(20,68){\footnotesize{$\beta_{2}$}}
\put(21,81){\footnotesize{$\vdots$}}

\put(40,28){\footnotesize{$\frac{ay_{0}\alpha_{1}}{x_{0}x_{1}}$}}
\put(40,48){\footnotesize{$\beta_{1}$}}
\put(40,68){\footnotesize{$\beta_{2}$}}
\put(41,81){\footnotesize{$\vdots$}}

\put(20,2){$(i)$}
\put(115,2){$(ii)$}

\pspolygon[linestyle=dashed,fillcolor=lightgray,fillstyle=crosshatch*,hatchcolor=white,hatchwidth=0.2pt,hatchsep=0.5pt](110,40)(110,83)(153,83)(153,40)
\psline[linewidth=1.5pt](90,20)(153,20)
\psline[linewidth=1.5pt](90,20)(90,83)
\psline(90,40)(110,40)
\psline[linewidth=0.75pt](110,40)(153,40)
\psline[linewidth=0.75pt](110,40)(110,83)

\put(120,10){$\rm{T}_1 \approx \sigma$}
\put(74,60){$\rm{T}_2\approx \tau$}
\put(115.5,60){$c(\rm{T}_1,\rm{T}_2) \approx \xi \times \eta$}

\put(98.5,41){\footnotesize{$a$}}

\end{picture}
\end{center}
\caption{Weight and Berger measure diagrams of a typical $2$-variable weighted shift in 
$\mathcal{TC}$.}
\label{FigureROMP}
\end{figure}

Let us now
denote the class of commuting pairs of subnormal operators on Hilbert space
by $\mathfrak{H}_{0}$, the class of subnormal pairs by $\mathfrak{H}_{\infty
}$, and for an integer $k\geq 1$, the class of $k$-hyponormal pairs in $%
\mathfrak{H}_{0}$ by $\mathfrak{H}_{k}$. \ Clearly, $\mathfrak{H}_{\infty
}\subseteq \cdots \subseteq \mathfrak{H}_{k}\subseteq \cdots \subseteq 
\mathfrak{H}_{1}\subseteq \mathfrak{H}_{0}$; the main results in \cite{CuYo1}
and \cite{CLY1} show that these inclusions are all proper; moreover,
examples illustrating these proper inclusions can be found in $\mathcal{TC}$%
. \ 

In this paper we show that for $\mathbf{T} \in \mathcal{TC}$, the
subnormality of \textit{any} power $\mathbf{T}^{(m,n)}$
implies the subnormality of $\mathbf{T}$. \ To
accomplish this, we first show that every power of $\mathbf{T} \in \mathcal{TC}$ 
is the orthogonal direct sum of $2$-variable
weighted shifts in $\mathcal{TC}$. \ Since each $2$-variable weighted shift
in $\mathcal{TC}$ is completely determined by five parameters, we then study how the five parameters
of each direct summand in a power are related to the five parameters in the
initial $2$-variable weighted shift. \ Next, we recall from \cite{ROMP} that each $\mathbf{T}\in \mathcal{TC}$ is associated with a pair
of linear functionals $\varphi \equiv \varphi (\mathbf{T})$ and $\psi \equiv \psi (\mathbf{T})$ (each
depending on the five parameters), and that $\mathbf{T}$ is subnormal if and only if $\varphi \geq 0$ and $\psi \geq 0$. \ With
all of this at our disposal, we proceed to establish a connection between
the pair $(\varphi ,\psi )$ associated with $\mathbf{T}$ and those associated with the summands in the orthogonal direct sum
decomposition of $\mathbf{T}^{(m,n)}$. \ This
eventually leads to the proof of our main result (Theorem \ref{Main 1}). \ 

This result provides a complete generalization of Theorem 3.9 in \cite{CLY2}%
. \ At the time we wrote \cite{CLY2}, the techniques available to us allowed
us to deal only with the quadratic powers $\mathbf{T}^{(1,2)}$ and 
$\mathbf{T}^{(2,1)}$; with the aid of a number of additional
examples, together with the main result in \cite{ROMP}, we have now been
able to handle the case of arbitrary powers. \ 

As an application of Theorem \ref{Main 1}, we can exhibit a hyponormal $2$-variable weighted shift such that  
none of its powers is subnormal. \ We describe the shift in Example \ref{Main 2}. \ This provides a striking and concrete example of the big gap
that exists between hyponormality and subnormality for $2$-variable weighted
shifts, even within a relatively simple class like $\mathcal{TC}$.

\section{\label{Notation}Notation and preliminaries}

To describe our results in detail we need some notation; we further expand on our
terminology later in this section. \ For $\mathbf{k}=(k_{1},k_{2})\in 
\mathbb{Z}_{+}^{2}$, we shall let $\mathcal{M}_{i}$ (respectively $\mathcal{N%
}_{j})$ be the subspace of $\ell ^{2}(\mathbb{Z}_{+}^{2})$ which is spanned
by the canonical orthonormal basis vectors associated to indices $\mathbf{k}$
with $k_{1}\geq 0$ and $k_{2}\geq i$ (resp. $k_{1}\geq j$ and $k_{2}\geq 0$%
). $\ $The \textit{core} of a $2$-variable weighted shift $\mathbf{T}$ is the restriction of $\mathbf{T}$
to $\mathcal{M}_{1}\cap \mathcal{N}_{1}$; in symbols, $c(\mathbf{T}):=\mathbf{T}|_{\mathcal{M}_{1}\cap 
\mathcal{N}_{1}}$. \ A $2$-variable weighted shift $\mathbf{T}$ is said to be of \textit{tensor form} if $\mathbf{T}\cong (I\otimes W_{\xi },W_{\eta }\otimes I)$, where $W_{\xi }$ and 
$W_{\eta }$ are unilateral weighted shifts. \ The class of all $2$-variable
weighted shifts $\mathbf{T} \in \mathfrak{H}_{0}$ whose
cores are of tensor form is denoted by $\mathcal{TC}$; that is, 

\begin{equation*}
\mathcal{TC}%
:=\{\mathbf{T} \in \mathfrak{H}_{0}:c(\mathbf{T}) \text{ is of tensor form}\}
\end{equation*}
(see Figure \ref{FigureROMP}(i)). \ 

It is well known that the commutativity of a pair of subnormals is necessary
but not sufficient for the existence of a lifting (\cite{Abr}, \cite{Lu1}, %
\cite{Lu2}, \cite{Lu3}), and it has recently been shown that the joint
hyponormality of the pair is necessary but not sufficient \cite{CuYo1}.\ \
Our previous work (\cite{CuYo1}, \cite{CuYo2}, \cite{CuYo3}, \cite{CuYo5}, %
\cite{CLY1}, \cite{CLY2}, \cite{ROMP}, \cite{CLY4}, \cite{Yo1}, \cite{Yo2})
has revealed that the nontrivial aspects of the LPCS are best detected
within the class $\mathfrak{H}_{0}$, especially within $\mathcal{TC}$; we
thus focus our attention on this class.

For a single operator $T$, the subnormality of all powers $T^{n}\;(n\geq 2)$
does not imply the subnormality of $T$, even if $T$ is a unilateral weighted
shift \cite[pp. 378-379]{Sta}. \ Thus one might guess that if we were to
impose a further condition such as the subnormality of a restriction of $T$
to an invariant subspace, e.g., the subnormality of $T|_{\mathbf{\vee }%
\{e_{k}\in \ell ^{2}(\mathbb{Z}_{+}):k\geq i\}}$ (for some $i\geq 1$), that $%
T$ would then be subnormal. \ However, even if we assume this for $i=1$, the
subnormality of $T$ is not guaranteed. \ For example, let $T:=\operatorname{shift}(%
\frac{1}{3},\frac{1}{2},1,1,\cdots )$. \ Then $T|_{\mathbf{\vee }\{e_{k}\in
\ell ^{2}(\mathbb{Z}_{+}):k\geq 1\}}\equiv \operatorname{shift}(\frac{1}{2}%
,1,1,\cdots )$ is subnormal, and also all powers $T^{n}\;(n\geq 2)$ are
subnormal, but $T$ is not subnormal. \ As a mater of fact, no backward
extension $\operatorname{shift}(\alpha _{0},\frac{1}{2},1,1,\cdots )$ can be
subnormal (\cite[Corollary 6]{QHWS}). \ More generally, the necessary and
sufficient conditions for a unilateral weighted shift $W_{\alpha }$ to be
subnormal when we assume that $W_{\alpha }|_{\mathbf{\vee }\{e_{k}\in \ell
^{2}(\mathbb{Z}_{+}):k\geq 1\}}$ is subnormal (with Berger measure $\mu $)
were obtained in \cite[Proposition 8]{QHWS}: $W_{\alpha }$ is subnormal if
and only if $\frac{1}{t}\in L^{1}(\mu )$ and $\alpha _{0}^{2}\left\| \frac{1%
}{t}\right\| _{L^{1}(\mu )}\leq 1$. \ 

In the multivariable case, the analogous results are highly nontrivial, if
one further assumes that each component is subnormal. \ In $1$-variable, the
subspace $\mathbf{\vee }\{e_{k}\in \ell ^{2}(\mathbb{Z}_{+}):k\geq 1\}$ can
be regarded as ``the core of $T$''; as we move into two variables it is
therefore natural to consider the condition $\mathbf{T} \in 
\mathcal{TC}$. \ 

To prove our results, we require a number of tools and techniques introduced
in previous work, e.g., the Six-point Test (Lemma \ref{joint hypo}), the
Backward Extension Theorem for $2$-variable weighted shifts (Lemma \ref%
{backext}) and the formula to reconstruct the Berger measure of a unilateral
weighted shift (\ref{reconstruct}), together with a new direct sum
decomposition for powers of $2$-variable weighted shifts which parallels the
decomposition used in \cite{CuP} to analyze $k$-hyponormality for powers of (%
$1$-variable) weighted shifts. \ Concretely, to analyze the power $\mathbf{T}^{(m,n)}$ of a $2$-variable weighted shift $\mathbf{T}%
\equiv (T_{1},T_{2})$, we split the ambient space 
$\ell ^{2}(\mathbb{Z}_{+}^{2})$ as an orthogonal direct sum $\oplus
_{p=0}^{m-1}\oplus _{q=0}^{n-1}\mathcal{H}_{(p,q)}^{(m,n)}$, where for $%
p=0,1,\cdots ,m-1,$ and $q=0,1,\cdots ,n-1,$%
\begin{equation}
\mathcal{H}_{(p,q)}^{(m,n)}:=\vee \{e_{(m\ell +p,nk+q)}:k,\ell \geq 0\}.
\label{hpq}
\end{equation}%
Each of the subspaces $\mathcal{H}_{(p,q)}^{(m,n)}$ reduces $T_{1}^{m}$ and $%
T_{2}^{n}$, and $\mathbf{T}^{(m,n)}$ is subnormal if
and only if each summand $\mathbf{T}^{(m,n)}|_{%
\mathcal{H}_{(p,q)}^{(m,n)}}$ is subnormal. \ For a set of pairs $\mathcal{X}
$, let $\bigoplus \mathcal{X}$ denote the set of pairs that can be written
as orthogonal sums of pairs in $\mathcal{X}$. \ We will show in Section \ref%
{Structure} that $\bigoplus \mathcal{TC}$ is invariant under the action $%
(m,n)\mapsto \mathbf{T}^{(m,n)}\;\;(m,n \geq 1)$.
\ 

Briefly stated, our strategy to prove our main result (Theorem \ref{Main 1})
is as follows: (i) if $\mathbf{T} \in \mathcal{TC}$ then each power 
$\mathbf{T} ^{(m,n)}\in \bigoplus \mathcal{TC}$; (ii) without loss
of generality, we can always assume $m=1$; (iii) the pair $(\varphi ,\psi )$
associated with $\mathbf{T}$ is directly related to the pairs $%
(\varphi ^{(p,q)},\psi ^{(p,q)})$ associated to the direct summands in the
orthogonal decomposition of $\mathbf{T}^{(1,n)}$; (iv) if a power $\mathbf{T}^{(1,n)}$ is subnormal, the functionals $\varphi
^{(0,0)} $ and $\psi ^{(0,1)}$ are both positive; and (v) it then follows
that $\varphi $ and $\psi $ are both positive, and therefore $\mathbf{T}$ is subnormal.

We devote the rest of this section to establishing some additional
terminology and notation. \ Let $\mathcal{H}$ be a complex Hilbert space and
let $\mathcal{B}(\mathcal{H})$ denote the algebra of bounded linear
operators on $\mathcal{H}$. \ Recall that a bounded linear operator $T\in 
\mathcal{B}(\mathcal{H})$ is \textit{normal} if $T^{\ast }T=TT^{\ast },$ and 
\textit{subnormal} if $T=N|_{\mathcal{H}}$, where $N$ is normal and $N(%
\mathcal{H})$ $\mathcal{\subseteq H}$. \ An operator $T$ is said to be 
\textit{hyponormal} if $T^{\ast }T\geq TT^{\ast }$. \ For $S,T\in \mathcal{B}%
(\mathcal{H})$, let $[S,T]:=ST-TS$. An $n$-tuple $\mathbf{T:}=(T_{1},\cdots
,T_{n})$ of operators on $\mathcal{H}$ is said to be (jointly) \textit{%
hyponormal} if the operator matrix 
\begin{equation*}
\lbrack \mathbf{T}^{\ast },\mathbf{T]:=}\left( 
\begin{array}{llll}
\lbrack T_{1}^{\ast },T_{1}] & [T_{2}^{\ast },T_{1}] & \cdots & [T_{n}^{\ast
},T_{1}] \\ 
\lbrack T_{1}^{\ast },T_{2}] & [T_{2}^{\ast },T_{2}] & \cdots & [T_{n}^{\ast
},T_{2}] \\ 
\text{ \thinspace \thinspace \quad }\vdots & \text{ \thinspace \thinspace
\quad }\vdots & \ddots & \text{ \thinspace \thinspace \quad }\vdots \\ 
\lbrack T_{1}^{\ast },T_{n}] & [T_{2}^{\ast },T_{n}] & \cdots & [T_{n}^{\ast
},T_{n}]%
\end{array}%
\right)
\end{equation*}%
is positive semidefinite on the direct sum of $n$ copies of $\mathcal{H}$
(cf. \cite{Ath}, \cite{CMX}, \cite{CP1}, \cite{CP2}, \cite{McCP}). \ For
instance, if $n=2 $, 
\begin{equation*}
\lbrack \mathbf{T}^{\ast },\mathbf{T]:=}\left( 
\begin{array}{ll}
\lbrack T_{1}^{\ast },T_{1}] & [T_{2}^{\ast },T_{1}] \\ 
\lbrack T_{1}^{\ast },T_{2}] & [T_{2}^{\ast },T_{2}]%
\end{array}%
\right) \text{.}
\end{equation*}%
The $n$-tuple $\mathbf{T}\equiv (T_{1},T_{2},\cdots ,T_{n})$ is said to be 
\textit{normal} if $\mathbf{T}$ is commuting and each $T_{i}$ is normal, and 
$\mathbf{T}$ is \textit{subnormal} if $\mathbf{T}$ is the restriction of a
normal $n$-tuple to a common invariant subspace. \ In particular, a
commuting pair $\mathbf{T}\equiv (T_{1},T_{2})$ is said to be $k$\textit{%
-hyponormal }$(k\geq 1)$ \cite{CLY1} if 
\begin{equation*}
\mathbf{T}(k):=(T_{1},T_{2},T_{1}^{2},T_{2}T_{1},T_{2}^{2},\cdots
,T_{1}^{k},T_{2}T_{1}^{k-1},\cdots ,T_{2}^{k})
\end{equation*}%
is hyponormal, or equivalently 
\begin{equation*}
\lbrack \mathbf{T}(k)^{\ast },\mathbf{T}(k)]=([(T_{2}^{q}T_{1}^{p})^{\ast
},T_{2}^{m}T_{1}^{n}])_{_{1\leq p+q\leq k}^{1\leq n+m\leq k}}\geq 0.
\end{equation*}%
Clearly, normal $\Rightarrow $ subnormal $\Rightarrow $ $k$-hyponormal. \
For $\alpha \equiv \{\alpha _{n}\}_{n=0}^{\infty }$ a bounded sequence of
positive real numbers (called \textit{weights}), let $W_{\alpha }:\ell ^{2}(%
\mathbb{Z}_{+})\rightarrow \ell ^{2}(\mathbb{Z}_{+})$ be the associated
unilateral weighted shift, defined by $W_{\alpha }e_{n}:=\alpha
_{n}e_{n+1}\;($all $n\geq 0)$, where $\{e_{n}\}_{n=0}^{\infty }$ is the
canonical orthonormal basis in $\ell ^{2}(\mathbb{Z}_{+}).$ \ For a weighted
shift $W_{\alpha }$, \textit{the moments of }$\alpha $ are given as 
\begin{equation}
\gamma _{k}\equiv \gamma _{k}(\alpha ):=\left\{ 
\begin{array}{cc}
1, & \text{if }k=0 \\ 
\alpha _{0}^{2}\cdots \alpha _{k-1}^{2}, & \text{if }k>0%
\end{array}%
\right. .  \label{moments}
\end{equation}%
It is easy to see that $W_{\alpha }$ is never normal, and that it is
hyponormal if and only if $\alpha _{0}\leq \alpha _{1}\leq \cdots $. \
Similarly, consider double-indexed positive bounded sequences $\alpha _{%
\mathbf{k}},\beta _{\mathbf{k}}\in \ell ^{\infty }(\mathbb{Z}_{+}^{2})$ , $%
\mathbf{k}\equiv (k_{1},k_{2})\in \mathbb{Z}_{+}^{2}$. \ We define the $2$%
-variable weighted shift $\mathbf{T}\equiv
(T_{1},T_{2})$ by 
\begin{equation*}
T_{1}e_{\mathbf{k}}:=\alpha _{\mathbf{k}}e_{\mathbf{k+}\varepsilon
_{1}}\quad \text{and}\quad T_{2}e_{\mathbf{k}}:=\beta _{\mathbf{k}}e_{%
\mathbf{k+}\varepsilon _{2}},
\end{equation*}%
where $\mathbf{\varepsilon }_{1}:=(1,0)$ and $\mathbf{\varepsilon }%
_{2}:=(0,1)$. \ Clearly, 
\begin{equation}
T_{1}T_{2}=T_{2}T_{1}\Longleftrightarrow \beta _{\mathbf{k+}\varepsilon
_{1}}\alpha _{\mathbf{k}}=\alpha _{\mathbf{k+}\varepsilon _{2}}\beta _{%
\mathbf{k}}\;\;(\text{all }\mathbf{k}\in \mathbb{Z}_{+}^{2}).
\label{commuting}
\end{equation}%
In an entirely similar way one can define multivariable weighted shifts. \
Given $\mathbf{k} \equiv (k_{1},k_{2}) \in \mathbb{Z}_{+}^{2}$, the \textit{%
moments} of $(\alpha ,\beta )$ of order $\mathbf{k}$ is 
\begin{equation*}
\gamma _{\mathbf{k}}\equiv \gamma _{\mathbf{k}}(\alpha ,\beta ):=%
\begin{cases}
1 & \text{if } k_{1}=0 \text{ and }k_{2}=0 \\ 
\alpha _{(0,0)}^{2}\cdots \alpha _{(k_{1}-1,0)}^{2} & \text{if }k_{1}\geq 1%
\text{ and }k_{2}=0 \\ 
\beta _{(0,0)}^{2}\cdots \beta _{(0,k_{2}-1)}^{2} & \text{if }k_{1}=0\text{
and }k_{2}\geq 1 \\ 
\alpha _{(0,0)}^{2}\cdots \alpha _{(k_{1}-1,0)}^{2}\cdot \beta
_{(k_{1},0)}^{2}\cdots \beta _{(k_{1},k_{2}-1)}^{2} & \text{if }k_{1}\geq 1%
\text{ and }k_{2}\geq 1.%
\end{cases}%
\end{equation*}%
We remark that, due to the commutativity condition (\ref{commuting}), $%
\gamma _{\mathbf{k}}$ can be computed using any nondecreasing path from $%
(0,0)$ to $(k_{1},k_{2})$. 

We now recall a well known characterization of
subnormality for multivariable weighted shifts \cite{JeLu} (due to C. Berger \cite[II.6.10]{Con} 
and independently established by R. Gellar and L.J. Wallen %
\cite{GeWa} in the single variable case): $\mathbf{T}$ admits a commuting normal extension if and only if there is a probability
measure $\mu $ (which we call the Berger measure of $\mathbf{T}$) defined on the $2$-dimensional rectangle $R=[0,a_{1}]\times
\lbrack 0,a_{2}]$ (where $a_{i}:=\left\| T_{i}\right\| ^{2}$) such that $%
\gamma _{\mathbf{k}}=\int_{R}s^{k_{1}}t^{k_{2}}d\mu (s,t),$ for all $\mathbf{%
k}\equiv (k_{1},k_{2})\in \mathbb{Z}_{+}^{2}$. \ 

The following well known result, which links the Berger measure of a
subnormal unilateral weighted shift with the Berger measure of its
restriction to a suitable invariant subspace, will be needed in Section \ref%
{Structure}. \ 

\begin{lemma}
\label{restrictionBerger}(\cite[p. 5140]{CuYo1}) \ Let $W_{\alpha }$ be a
subnormal unilateral weighted shift and let $\xi $ denote its Berger
measure. \ For $n\geq 1$ let $\mathcal{L}_{n}:=\bigvee \{e_{h}:h\geq n\}$
denote the invariant subspace obtained by removing the first $n$ vectors in
the canonical orthonormal basis of $\ell ^{2}(\mathbb{Z}_{+})$. $\ $Then the
Berger measure of $W_{\alpha }|_{\mathcal{L}_{n}}$ is 
\begin{equation}
d\xi _{n}(s):=\frac{s^{n}}{\gamma _{n}}d\xi (s),  \label{Berger}
\end{equation}%
where $\gamma _{n}$ is the $n$-th moment of $\alpha $, given by (\ref%
{moments}). \ 
\end{lemma}

We will occasionally write $\operatorname{shift}(\alpha _{0},\alpha _{1},\cdots )$
to denote the weighted shift with weight sequence $\{\alpha
_{k}\}_{k=0}^{\infty }$. \ We also denote by $U_{+}:=\operatorname{shift}(1,1,\cdots
)$ the (unweighted) unilateral shift, and for $0<a<1$ we let $S_{a}:=\operatorname{%
shift}(a,1,1,\cdots )$. \ Observe that $U_{+}$ and $S_{a}$ are subnormal,
with Berger measures $\delta _{1}$ and $(1-a^{2})\delta _{0}+a^{2}\delta
_{1} $, respectively, where $\delta _{p}$ denotes the point-mass probability
measure with support the singleton $\{p\}$.

Given integers $i$ and $m$ $(m\geq 1,0\leq i\leq m-1),$ consider $\mathcal{H}%
\equiv \ell ^{2}(\mathbb{Z}_{+})=\vee \{e_{n}:n\geq 0\}$ and define $%
\mathcal{H}_{i}:=\vee \{e_{mj+i}:j\geq 0\},$ so $\mathcal{H}=\oplus
_{i=0}^{m-1}\mathcal{H}_{i}.$ \ For a sequence $\alpha \equiv \{\alpha
_{n}\}_{n=0}^{\infty }$, let $\alpha (m:i):=\{\Pi _{k=0}^{m-1}\alpha
_{mj+i+k}\}_{j=0}^{\infty }$, that is, $\alpha (m:i)$ denotes the sequence
of products of numbers in adjacent packets of size $m,$ beginning with the product $%
\alpha _{i}\cdots \alpha _{i+m-1}.$ \ For example, $\alpha (2:0):\alpha
_{0}\alpha _{1},\alpha _{2}\alpha _{3},\alpha _{4}\alpha _{5},\cdots $, and $%
\alpha (3:2):\alpha _{2}\alpha _{3}\alpha _{4},\alpha _{5}\alpha _{6}\alpha
_{7},\cdots $. \ Then for $m\geq 1$ and $0\leq i\leq m-1$, $W_{\alpha (m:i)}$
is unitarily equivalent to $W_{\alpha }^{m}|_{\mathcal{H}_{i}}.$ \
Therefore, $W_{\alpha }^{m}$ is unitarily equivalent to $\oplus
_{i=0}^{m-1}W_{\alpha (m:i)}.$ \ This analysis naturally leads to the
following result, which will be needed in Section \ref{pairs}.

\begin{lemma}
\label{powermeasure}(\cite[Theorem 2.9]{CuP}) \ Let $W_{\alpha }$ be a
subnormal unilateral weighted shift with Berger measure $\mu $. \ Then $%
W_{\alpha (m,i)}$ is subnormal with Berger measure $\mu _{(m,i)}$, where 
\begin{equation*}
\begin{tabular}{l}
$d\mu _{(m,0)}(s)=d\mu (s^{\frac{1}{m}})$ and $d\mu _{(m,i)}(s)=\frac{s^{%
\frac{i}{m}}}{\gamma _{i}}d\mu (s^{\frac{1}{m}})$ for $1\leq i\leq m-1$ .%
\end{tabular}%
\end{equation*}
\end{lemma}

\section{Statement of the main result}

In (\cite{CLY2}) we showed that if $\mathbf{T}\in 
\mathcal{TC}$ (see Figure \ref{FigureROMP}(i)), 
\begin{equation}
\mathbf{T}^{(1,2)}\in \mathfrak{H}_{\infty
}\Leftrightarrow \mathbf{T}^{(2,1)}\in \mathfrak{H}%
_{\infty }\Leftrightarrow \mathbf{T}\in \mathfrak{H}%
_{\infty }.  \label{previous}
\end{equation}%
The main result in this paper, which follows, is based on (\ref{previous})
and a myriad of examples that have arisen in our previous research.

\begin{theorem}
\label{Main 1} Let $\mathbf{T}\in \mathcal{TC}$. \
The following statements are equivalent.\newline
(i) $\mathbf{T}\in \mathfrak{H}_{\infty }$;\newline
(ii) $\mathbf{T}^{(m,n)}\in \bigoplus \mathfrak{H}%
_{\infty }$ for all $m,n\geq 1$;\newline
(iii) $\mathbf{T}^{(m,n)}\in \bigoplus \mathfrak{H}%
_{\infty }$ for some $m,n\geq 1$.
\end{theorem}

\section{Some basic facts}

For the reader's convenience, in this section we list several well known
auxiliary results and definitions which are needed for the proof of the
main result. \ First, to detect hyponormality for $2$%
-variable weighted shifts we use a simple criterion involving a base point $%
\mathbf{k}$ in $\mathbb{Z}_{+}^{2}$ and its five neighboring points in $%
\mathbf{k}+\mathbb{Z}_{+}^{2}$ at path distance at most $2$.

\begin{lemma}
(\cite{bridge})\label{joint hypo}(Six-point Test) \ Let $\mathbf{T}\equiv (T_{1},T_{2})$ be a $2$-variable weighted shift, with
weight sequences $\alpha $ and $\beta $. \ Then 
\begin{equation*}
\begin{tabular}{l}
$\lbrack \mathbf{T}^{\ast },\mathbf{T}\mathbf{]\geq }0 \Longleftrightarrow$ \\ 
\\ 
$H_{(k_{1},k_{2})}(1):=\left( 
\begin{array}{cc}
\alpha _{\mathbf{k}+\mathbf{\varepsilon }_{1}}^{2}-\alpha _{\mathbf{k}}^{2}
& \alpha _{\mathbf{k}+\mathbf{\varepsilon }_{2}}\beta _{\mathbf{k}+\mathbf{%
\varepsilon }_{1}}-\alpha _{\mathbf{k}}\beta _{\mathbf{k}} \\ 
\alpha _{\mathbf{k}+\mathbf{\varepsilon }_{2}}\beta _{\mathbf{k}+\mathbf{%
\varepsilon }_{1}}-\alpha _{\mathbf{k}}\beta _{\mathbf{k}} & \beta _{\mathbf{%
k}+\mathbf{\varepsilon }_{2}}^{2}-\beta _{\mathbf{k}}^{2}%
\end{array}%
\right) \geq 0 \ (\text{all } \mathbf{k}\in \mathbb{Z}_{+}^{2})%
\text{.}$%
\end{tabular}%
\end{equation*}
\end{lemma}

Next, we present a criterion to detect the subnormality of $2$-variable
weighted shifts. \ First, we need some definitions.

(i) Let $\mu $ and $\nu $ be two positive measures on $\mathbb{R}_{+}.$ \ We
say that $\mu \leq \nu $ on $X:=\mathbb{R}_{+},$ if $\mu (E)\leq \nu (E)$
for all Borel subset $E\subseteq \mathbb{R}_{+}$; equivalently, $\mu \leq
\nu $ if and only if $\int fd\mu \leq \int fd\nu $ for all $f\in C(X)$ such
that $f\geq 0$ on $\mathbb{R}_{+}$.

(ii)\ Let $\mu $ be a probability measure on $X\times Y$, and assume that $%
\frac{1}{t}\in L^{1}(\mu ).$ \ The \textit{extremal measure} $\mu _{ext}$
(which is also a probability measure) on $X\times Y$ is given by $d\mu
_{ext}(s,t):=(1-\delta _{0}(t))\frac{1}{t\left\| \frac{1}{t}\right\|
_{L^{1}(\mu )}}d\mu (s,t)$.

(iii) Given a measure $\mu $ on $X\times Y$, the \textit{marginal measure} $%
\mu ^{X}$ is given by $\mu ^{X}:=\mu \circ \pi _{X}^{-1}$, where $\pi
_{X}:X\times Y\rightarrow X$ is the canonical projection onto $X$. \ Thus $%
\mu ^{X}(E)=\mu (E\times Y)$, for every $E\subseteq X$.

To state the following result, recall the notation in (\ref{hpq}), and let $%
\mathcal{M}:=\mathcal{M}_1 \equiv \mathcal{H}_{(0,1)}^{(1,1)}$. \ 

\begin{lemma}
\label{backext}(\cite[Proposition 3.10]{CuYo1}) \ (Subnormal backward
extension) \ Let $\mathbf{T}\equiv (T_{1},T_{2})$ be
a $2$-variable weighted shift, and assume that $\mathbf{T}\mathbf{|}_{\mathcal{M}}$ is subnormal with associated measure $\mu
_{\mathcal{M}}$ and that $W_{0}:=\operatorname{shift}(\alpha _{00},\alpha
_{10},\cdots )$ is subnormal with associated measure $\sigma $. \ Then $\mathbf{T}$ is subnormal if and only if\newline
(i) $\ \frac{1}{t}\in L^{1}(\mu _{\mathcal{M}})$;\newline
(ii) $\ \beta _{00}^{2}\leq (\left\| \frac{1}{t}\right\| _{L^{1}(\mu _{%
\mathcal{M}})})^{-1}$;\newline
(iii) $\ \beta _{00}^{2}\left\| \frac{1}{t}\right\| _{L^{1}(\mu _{\mathcal{M}%
})}(\mu _{\mathcal{M}})_{ext}^{X}\leq \sigma $.

Moreover, if $\beta _{00}^{2}\left\| \frac{1}{t}\right\| _{L^{1}(\mu _{%
\mathcal{M}})}=1,$ then $(\mu _{\mathcal{M}})_{ext}^{X}=\sigma $. \ In the
case when $\mathbf{T}$ is subnormal, its Berger
measure $\mu $ is given by 
\begin{eqnarray}
d\mu (s,t) &=&\beta _{00}^{2}\left\| \frac{1}{t}\right\| _{L^{1}(\mu _{%
\mathcal{M}})}d(\mu _{\mathcal{M}})_{ext}(s,t)  \notag \\
&&+\left( d\sigma (s)-\beta _{00}^{2}\left\| \frac{1}{t}\right\| _{L^{1}(\mu
_{\mathcal{M}})}d(\mu _{\mathcal{M}})_{ext}^{X}(s)\right) d\delta _{0}(t).
\label{reconstruct}
\end{eqnarray}
\end{lemma}

\section{\label{Structure}The structure of powers of $2$-variable weighted shifts in $\mathcal{TC}$}

Consider a $2$-variable weighted shift $\mathbf{T}\equiv (T_{1},T_{2})\in \mathcal{TC}$ (see Figure \ref{FigureROMP}(i)). \
Since $T_{1}$ (resp. $T_{2}$) is subnormal, we know that $\operatorname{shift}%
(\alpha _{1},\alpha _{2},\alpha _{3}\cdots )$ (resp. $\operatorname{shift}(\beta
_{1},\beta _{2},\beta _{3}\cdots )$) is subnormal; let $\xi $ (resp. $\eta $%
) be its Berger measure. \ Similarly, let $\sigma $ (resp. $\tau $)
denote the Berger measure of $\operatorname{shift}(x_{0},x_{1},x_{2},\cdots )$
(resp. $\operatorname{shift}(y_{0},y_{1},y_{2},\cdots )$). \ Finally, let $\tau
_{1}$ be the Berger measure of $\operatorname{shift}(y_{1},y_{2},y_{3},\cdots )$ $%
\equiv \operatorname{shift}(y_{0},y_{1},y_{2},\cdots )|_{\vee \{e_{k}:k\geq 1\}}$. \
Figure \ref{FigureROMP} shows the general form of a pair in $\mathcal{TC}$,
and that it is uniquely determined by the five parameters $\sigma $, $\tau $%
, $a$, $\xi $ and $\eta $. \ Thus, in what follows we will identify a pair $\mathbf{T}\in \mathcal{TC}$ with the $5$-tuple $\left\langle
\sigma ,\tau ,a,\xi ,\eta \right\rangle $. \ We shall also let $\left[ a,\xi %
\right] $ denote the Berger measure of the subnormal unilateral weighted
shift $W$ whose $0$-th weight is $a$ and with $\xi $ as the Berger measure
of $W|_{\mathcal{L}_{1}}$, where $\mathcal{L}_{1}:=\bigvee \{e_{k}:k\geq 1\}$%
. \ For instance, in Figure \ref{FigureROMP}(i) the Berger measure for the first
row is $\left[ a,\xi \right] $, and for the second row is $\left[ \frac{a \beta_1}{y_1},\xi \right] $. \ Finally, we shall let $z_{j}\equiv
z_{j}(\eta )$ denote the $j$-th weight of the unilateral weighted shift
whose Berger measure is $\eta $; that is, $\operatorname{shift}(z_{0},z_{1},\cdots )$
has Berger measure $\eta $.

\begin{lemma}
\label{LemmaStructure} \ Let $\mathbf{T}\in \mathcal{TC}$, and let 
$m,n\geq 1$. \ Then $\mathbf{T}^{(m,n)}\in \bigoplus \mathcal{TC}$.
\end{lemma}

\begin{proof}
First recall from (\ref{hpq}) that $\ell ^{2}(\mathbb{Z}_{+}^{2})$ can be
written as an orthogonal direct sum $\oplus _{p=0}^{m-1}\oplus _{q=0}^{n-1}%
\mathcal{H}_{(p,q)}^{(m,n)}$, where for $p=0,1,\cdots,m-1$ and $%
q=0,1,\cdots,n-1$ we have $\mathcal{H}_{(p,q)}^{(m,n)}:=\vee \{e_{(m\ell
+p,nk+q)}:k,\ell \geq 0\}$. Now write $\mathbf{T}\equiv
\left\langle \sigma ,\tau ,a,\xi ,\eta \right\rangle $. \ We shall establish
that 
\begin{equation}
\mathbf{T}^{(m,n)}=\bigoplus_{p=0}^{m-1}\bigoplus_{q=0}^{n-1}\left\langle \sigma
^{(p,q)},\tau ^{(p,q)},a^{(p,q)},\xi ^{(p,q)},\eta ^{(p,q)}\right\rangle ,
\label{structureequation}
\end{equation}%
where $\left\langle \sigma ^{(p,q)},\tau ^{(p,q)},a^{(p,q)},\xi
^{(p,q)},\eta ^{(p,q)}\right\rangle$ is the 5-tuple associated to the
restriction of $\mathbf{T}^{(m,n)}$ to the reducing subspace $%
\mathcal{H}_{(p,q)}^{(m,n)}$. \ Since $\mathbf{T}^{(m,n)}=(\mathbf{T}^{(m,1)})^{(1,n)}$, and since $\mathbf{T}^{(m,1)}\in \bigoplus \mathcal{TC}$ if and only if $\mathbf{T}^{(1,m)}\in \bigoplus \mathcal{TC}$, it suffices to prove (\ref%
{structureequation}) in the case $m=1$. \ The proof is simple but a bit
tedious, and it entails careful diagram chasing in Figure \ref{double}. \
Visual inspection of that Figure reveals that when $m=1$ we have $\sigma
^{(0,0)}=\sigma $, $\sigma ^{(0,1)}=\left[ a,\xi \right] $, $\cdots $ , $%
\sigma ^{(0,n-1)}=\left[ \frac{az_{0}\cdots z_{n-3}}{y_{1}\cdots y_{n-2}}%
,\xi \right] \;(n\geq 3)$; moreover, $\tau ^{(0,q)}=\tau _{(n,q)}$, where
the latter notation was introduced in Lemma \ref{powermeasure}. \ Still
looking at Figure \ref{FigureROMP}, we observe that $a^{(0,q)}=\frac{%
az_{0}\cdots z_{n-2+q}}{y_{1}\cdots y_{n-1+q}}$ and that $\xi ^{(0,q)}=\xi$,
$\eta ^{(0,0)}=\eta_{(n,n-1)}$ and $\eta^{(0,q)}=(\eta_{1}) _{(n,q-1)} \; (q \geq 1)$. \ This completes the proof. \qed
\end{proof}

\begin{corollary}
For $m,n\geq 1$ we have $(\bigoplus \mathcal{TC})^{(m,n)}\subseteq \bigoplus 
\mathcal{TC}$.
\end{corollary}

We now restate the main result in \cite{ROMP}. \ First, recall that if $%
\tau $ is the Berger measure of $\operatorname{shift}(y_{0},y_{1},\cdots )$, we
denote by $\tau _{1}$ the Berger measure of $\operatorname{shift}(y_{1},y_{2},\cdots
)$. \ As described in Lemma \ref{restrictionBerger}, we know that $d\tau
_{1}(t)\equiv \frac{t}{y_{0}^{2}}d\tau (t)$.

\begin{lemma}
\label{LemmaROMP}(\cite[Theorem 2.3]{ROMP}) \ Let $\mathbf{T}\equiv \left\langle \sigma ,\tau ,a,\xi ,\eta \right\rangle \in \mathcal{TC}$ be as in Figure \ref{FigureROMP2}(i) and let 
\begin{eqnarray}
\psi &:&=\tau _{1}-a^{2}\left\| \frac{1}{s}\right\| _{L^{1}(\xi )}\eta
\label{formulas} \\
\varphi &:&=\sigma -y_{0}^{2}\left\| \frac{1}{t}\right\| _{L^{1}(\psi
)}\delta _{0}-a^{2}y_{0}^{2}\left\| \frac{1}{t}\right\| _{L^{1}(\eta )}\frac{%
\xi }{s},  \notag
\end{eqnarray}%
where $y_{0}\equiv \beta _{00}:=\sqrt{\int t\;d\tau (t)}$. \ Then $\mathbf{T}$ is subnormal if and only if $\psi \geq 0$ and $\varphi
\geq 0$.
\end{lemma}

\setlength{\unitlength}{.85mm} \psset{unit=.85mm} 
\begin{figure}[th]
\begin{center}
\begin{picture}(165,85)

\pspolygon[linestyle=dashed,fillcolor=lightgray,fillstyle=crosshatch*,hatchcolor=white,hatchwidth=0.2pt,hatchsep=0.5pt](20,40)(20,83)(63,83)(63,40)
\psline[linewidth=1.5pt](0,20)(63,20)
\psline[linewidth=1.5pt](0,20)(0,83)
\psline(0,40)(20,40)
\psline[linewidth=0.75pt](20,40)(63,40)
\psline[linewidth=0.75pt](20,40)(20,83)

\put(10,2){$(i)$}

\put(30,10){$\rm{T}_1 \approx \sigma$}
\put(-16,60){$\rm{T}_2\approx \tau$}
\put(25.5,60){$c(\rm{T}_1,\rm{T}_2) \approx \xi \times \eta$}

\put(9,41){\footnotesize{$a$}}


\put(90,2){$(ii)$}

\psline{->}(110,13)(130,13)
\put(120,8){$\rm{T}_1$}
\psline{->}(87,50)(87,70)
\put(81,60){$\rm{T}_2$}

\psline{->}(90,20)(155,20)
\psline(90,40)(153,40)
\psline(90,60)(153,60)
\psline(90,80)(153,80)

\psline{->}(90,20)(90,85)
\psline(110,20)(110,83)
\psline(130,20)(130,83)
\psline(150,20)(150,83)

\put(83,16){\footnotesize{$(0,0)$}}
\put(106,16){\footnotesize{$(1,0)$}}
\put(126,16){\footnotesize{$(2,0)$}}
\put(146,16){\footnotesize{$(3,0)$}}

\put(97,21){\footnotesize{$x$}}
\put(117,21){\footnotesize{$1$}}
\put(137,21){\footnotesize{$1$}}
\put(151,21){\footnotesize{$\cdots$}}

\put(97,41){\footnotesize{$a$}}
\put(117,41){\footnotesize{$1$}}
\put(137,41){\footnotesize{$1$}}
\put(151,41){\footnotesize{$\cdots$}}

\put(97,61){\footnotesize{$a$}}
\put(117,61){\footnotesize{$1$}}
\put(137,61){\footnotesize{$1$}}
\put(151,61){\footnotesize{$\cdots$}}

\put(97,81){\footnotesize{$\cdots$}}
\put(117,81){\footnotesize{$\cdots$}}
\put(137,81){\footnotesize{$\cdots$}}
\put(151,81){\footnotesize{$\cdots$}}

\put(90,28){\footnotesize{$y$}}
\put(90,48){\footnotesize{$1$}}
\put(90,68){\footnotesize{$1$}}
\put(91,81){\footnotesize{$\vdots$}}

\put(110,28){\footnotesize{$1$}}
\put(110,48){\footnotesize{$1$}}
\put(110,68){\footnotesize{$1$}}
\put(111,81){\footnotesize{$\vdots$}}

\put(130,28){\footnotesize{$1$}}
\put(130,48){\footnotesize{$1$}}
\put(130,68){\footnotesize{$1$}}
\put(131,81){\footnotesize{$\vdots$}}

\end{picture}
\end{center}
\caption{Berger measure diagram and weight diagram of the 2-variable weighted shifts in Lemma \ref%
{LemmaROMP} and Example \ref{simpleex}, respectively.}
\label{FigureROMP2}
\end{figure}
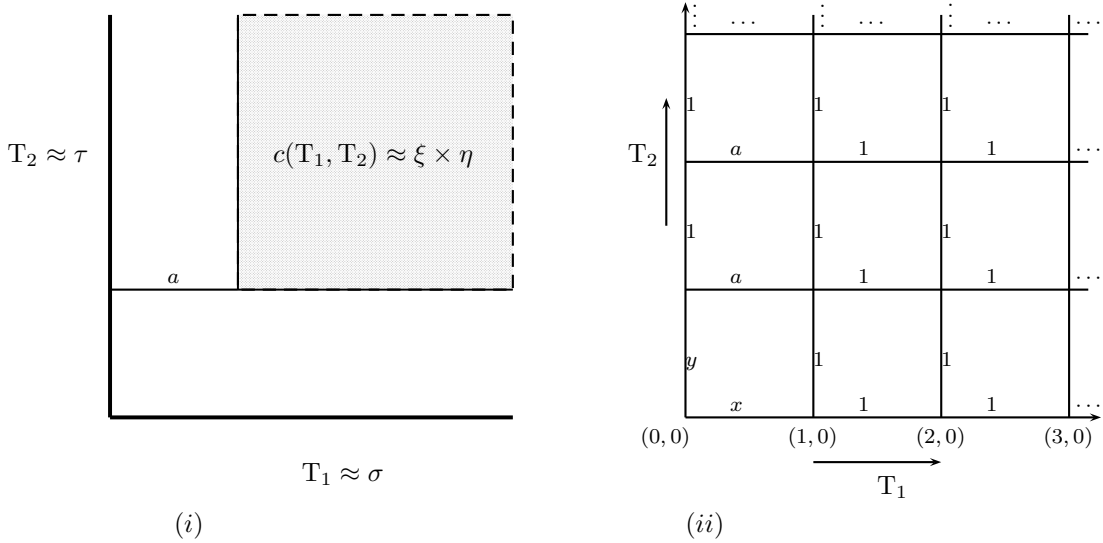

\begin{example}
\label{simpleex}Assume the very simple case of $\mathbf{T}\equiv
\left\langle \sigma ,\tau ,a,\xi ,\eta \right\rangle $, where $\sigma :=%
\left[ x,\delta _{1}\right] $, $\tau :=\left[ y,\delta _{1}\right] $, $0<a<1$%
, $\xi :=\delta _{1}$ and $\eta :=\delta _{1}$ (cf. Figure \ref{FigureROMP2}%
(ii)). \ Then $\psi =\delta _{1}-a^{2}\delta _{1}=(1-a^{2})\delta _{1}$ and $%
\varphi =(1-x^{2})\delta _{0}+x^{2}\delta _{1}-y^{2}(1-a^{2})\delta
_{0}-a^{2}y^{2}\delta _{1}=\{(1-x^{2})-y^{2}(1-a^{2})\}\delta
_{0}+(x^{2}-a^{2}y^{2})\delta _{1}$. \ Thus, $\mathbf{T}$ is
subnormal if and only if $(1-x^{2})-y^{2}(1-a^{2})\geq 0$, a condition
identical to that in \cite[Proposition 2.11]{CuYo1}.
\end{example}

\section{\label{pairs}The pairs $(\protect\psi ,\protect\varphi )$ for $%
\left\langle \protect\sigma ,\protect\tau ,a,\protect\xi ,\protect\eta %
\right\rangle $ and $\left\langle \protect\sigma ^{(p,q)},\protect\tau %
^{(p,q)},a^{(p,q)},\protect\xi ^{(p,q)},\protect\eta ^{(p,q)}\right\rangle $}

In this section we establish a direct relationship between the pair $(\psi
,\varphi )$ associated to $\left\langle \sigma ,\tau ,a,\xi ,\eta
\right\rangle \in \mathcal{TC}$ and some of the pairs $(\psi
^{(p,q)},\varphi ^{(p,q)})$ associated to the direct summands $\left\langle
\sigma ^{(p,q)},\tau ^{(p,q)},a^{(p,q)},\xi ^{(p,q)},\eta
^{(p,q)}\right\rangle $ in $\left\langle \sigma ,\tau ,a,\xi ,\eta
\right\rangle ^{(m,n)}$.

\begin{proposition}
\label{connection}Let $\left\langle \sigma ,\tau ,a,\xi ,\eta \right\rangle
\in \mathcal{TC}$, and let $n\geq 2$. \ Consider the decomposition $%
\left\langle \sigma ,\tau ,a,\xi ,\eta \right\rangle
^{(1,n)}=\bigoplus_{q=0}^{n-1}\left\langle \sigma ^{(0,q)},\tau
^{(0,q)},a^{(0,q)},\xi ^{(0,q)},\eta ^{(0,q)}\right\rangle $, and let $(\psi
,\varphi )$ (resp. $(\psi ^{(0,q)},\varphi ^{(0,q)})$) be the associated
pair in Lemma \ref{LemmaROMP}. \ Then \newline
(i) \ $\psi ^{(0,1)}\geq 0\Longleftrightarrow \psi \geq 0$; and\newline
(ii) \ $\varphi ^{(0,0)}\geq 0\Longleftrightarrow \varphi \geq 0$.
\end{proposition}

\setlength{\unitlength}{.91mm} \psset{unit=.91mm} 
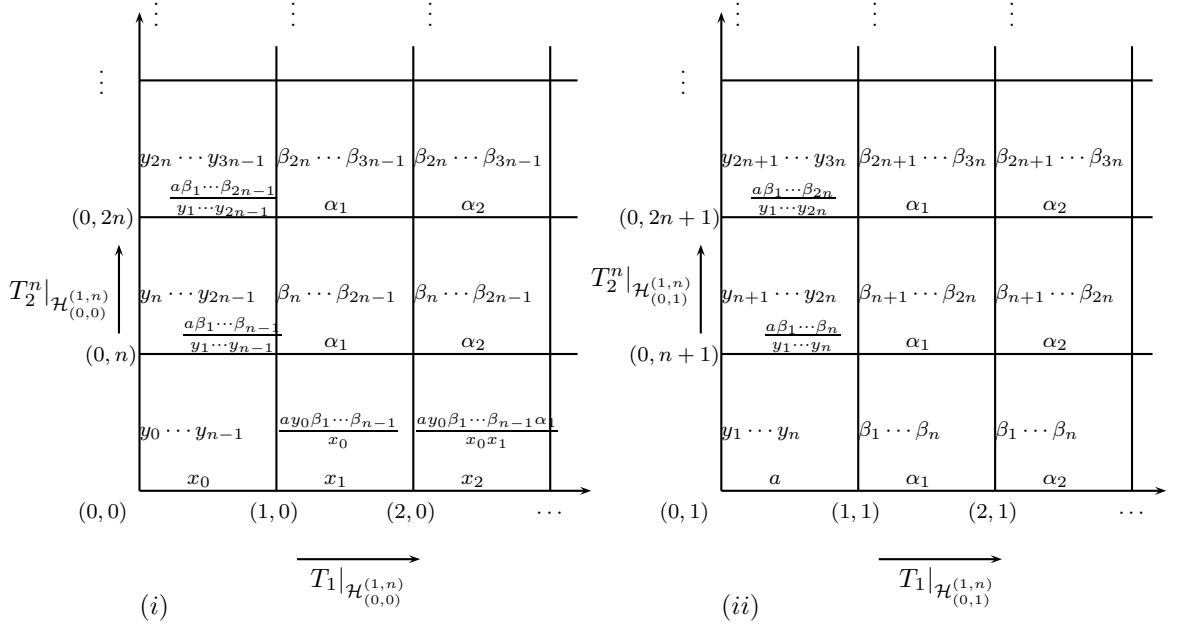
\begin{figure}[th]
\begin{center}
\begin{picture}(160,90)

\psline{->}(10,20)(76,20)
\psline(10,40)(74,40)
\psline(10,60)(74,60)
\psline(10,80)(74,80)
\psline{->}(10,20)(10,90)
\psline(30,20)(30,85)
\psline(50,20)(50,85)
\psline(70,20)(70,85)

\put(1,16){\footnotesize{$(0,0)$}}
\put(26,16){\footnotesize{$(1,0)$}}
\put(46,16){\footnotesize{$(2,0)$}}
\put(68,16){\footnotesize{$\cdots$}}

\put(17,21){\footnotesize{$x_{0}$}}
\put(37,21){\footnotesize{$x_{1}$}}
\put(57,21){\footnotesize{$x_{2}$}}

\put(16,42){\footnotesize{$\frac{a\beta_{1}\cdots \beta_{n-1}}{y_{1}\cdots y_{n-1}}$}}
\put(37,41){\footnotesize{$\alpha_{1}$}}
\put(57,41){\footnotesize{$\alpha_{2}$}}

\put(14,62){\footnotesize{$\frac{a\beta_{1}\cdots \beta_{2n-1}}{y_{1}\cdots y_{2n-1}}$}}
\put(37,61){\footnotesize{$\alpha_{1}$}}
\put(57,61){\footnotesize{$\alpha_{2}$}}

\psline{->}(33,10)(51,10)
\put(35,6){$T_1|_{\mathcal{H}_{(0,0)}^{(1,n)}}$}

\put(2,39){\footnotesize{$(0,n)$}}
\put(0,59){\footnotesize{$(0,2n)$}}
\put(4,78){\footnotesize{$\vdots$}}

\psline{->}(7,43)(7,56)
\put(-9,48){$T_{2}^n|_{\mathcal{H}_{(0,0)}^{(1,n)}}$}

\put(10,28){\footnotesize{$y_{0} \cdots y_{n-1}$}}
\put(10,48){\footnotesize{$y_{n} \cdots y_{2n-1}$}}
\put(10,68){\footnotesize{$y_{2n} \cdots y_{3n-1}$}}
\put(12,88){\footnotesize{$\vdots$}}

\put(30,28){\footnotesize{$\frac{ay_{0}\beta_{1}\cdots\beta_{n-1}}{x_{0}}$}}
\put(30,48){\footnotesize{$\beta_{n}\cdots\beta_{2n-1}$}}
\put(30,68){\footnotesize{$\beta_{2n}\cdots\beta_{3n-1}$}}
\put(32,88){\footnotesize{$\vdots$}}

\put(50,28){\footnotesize{$\frac{ay_{0}\beta_{1}\cdots\beta_{n-1}\alpha_{1}}{x_{0}x_{1}}$}}
\put(50,48){\footnotesize{$\beta_{n}\cdots\beta_{2n-1}$}}
\put(50,68){\footnotesize{$\beta_{2n}\cdots\beta_{3n-1}$}}
\put(52,88){\footnotesize{$\vdots$}}

\put(10,2){$(i)$}


\put(95,2){$(ii)$}

\psline{->}(95,20)(160,20)
\psline(95,40)(158,40)
\psline(95,60)(158,60)
\psline(95,80)(158,80)
\psline{->}(95,20)(95,90)
\psline(115,20)(115,85)
\psline(135,20)(135,85)
\psline(155,20)(155,85)

\put(86,16){\footnotesize{$(0,1)$}}
\put(111,16){\footnotesize{$(1,1)$}}
\put(131,16){\footnotesize{$(2,1)$}}
\put(153,16){\footnotesize{$\cdots$}}

\put(102,21){\footnotesize{$a$}}
\put(122,21){\footnotesize{$\alpha_{1}$}}
\put(142,21){\footnotesize{$\alpha_{2}$}}

\put(101,42){\footnotesize{$\frac{a\beta_{1}\cdots\beta_{n}}{y_{1}\cdots y_{n}}$}}
\put(122,41){\footnotesize{$\alpha_{1}$}}
\put(142,41){\footnotesize{$\alpha_{2}$}}

\put(99,62){\footnotesize{$\frac{a\beta_{1}\cdots\beta_{2n}}{y_{1}\cdots y_{2n}}$}}
\put(122,61){\footnotesize{$\alpha_{1}$}}
\put(142,61){\footnotesize{$\alpha_{2}$}}

\psline{->}(118,10)(136,10)
\put(121,6){$T_1|_{\mathcal{H}_{(0,1)}^{(1,n)}}$}

\put(81.5,39){\footnotesize{$(0,n+1)$}}
\put(80,59){\footnotesize{$(0,2n+1)$}}
\put(89,78){\footnotesize{$\vdots$}}

\psline{->}(92,43)(92,56)
\put(76,50){$T_{2}^n|_{\mathcal{H}_{(0,1)}^{(1,n)}}$}

\put(95,28){\footnotesize{$y_{1}\cdots y_{n}$}}
\put(95,48){\footnotesize{$y_{n+1}\cdots y_{2n}$}}
\put(95,68){\footnotesize{$y_{2n+1}\cdots y_{3n}$}}
\put(97,88){\footnotesize{$\vdots$}}

\put(115,28){\footnotesize{$\beta_{1}\cdots\beta_{n}$}}
\put(115,48){\footnotesize{$\beta_{n+1}\cdots\beta_{2n}$}}
\put(115,68){\footnotesize{$\beta_{2n+1}\cdots\beta_{3n}$}}
\put(117,88){\footnotesize{$\vdots$}}

\put(135,28){\footnotesize{$\beta_{1}\cdots\beta_{n}$}}
\put(135,48){\footnotesize{$\beta_{n+1}\cdots\beta_{2n}$}}
\put(135,68){\footnotesize{$\beta_{2n+1}\cdots\beta_{3n}$}}
\put(137,88){\footnotesize{$\vdots$}}

\end{picture}
\end{center}
\caption{Weight diagrams of $\mathbf{T}^{(1,n)}|_{%
\mathcal{H}_{(0,0)}^{(1,n)}}$ and $\mathbf{T}^{(1,n)}|_{\mathcal{H}_{(0,1)}^{(1,n)}}$.}
\label{double}
\end{figure}

\begin{proof}
We refer the reader to Figure \ref{double}. \ By Lemma \ref{LemmaStructure},
we have $\sigma ^{(0,0)}=\sigma $, $\sigma ^{(0,1)}=\left[ a,\xi \right] $, $%
\tau ^{(0,0)}=\tau _{(n,0)}$, $\tau ^{(0,1)}=\tau _{(n,1)}$, $a^{(0,0)}=%
\frac{az_{0}\cdots z_{n-2}}{y_{1}\cdots y_{n-1}}$, $a^{(0,1)}=\frac{%
az_{0}\cdots z_{n-1}}{y_{1}\cdots y_{n}}$, $\xi ^{(0,0)}=\xi ^{(0,1)}=\xi $, 
$\eta ^{(0,0)}=\eta _{(n,n-1)}$ and $\eta ^{(0,1)}=(\eta
_{(n,0)})_{1}$. \ Then 
\begin{eqnarray}
\psi ^{(0,1)} &=&(\tau ^{(0,1)})_{1}-(a^{(0,1)})^{2}\left\| \frac{1}{s}%
\right\| _{L^{1}(\xi ^{(0,1)})}\eta ^{(0,1)}  \label{eq1} \\
&=&(\tau _{(n,1)})_{1}-\frac{a^{2}z_{0}^{2}\cdots z_{n-1}^{2}}{%
y_{1}^{2}\cdots y_{n}^{2}}\left\| \frac{1}{s}\right\| _{L^{1}(\xi )}(\eta
_{(n,0)})_{1}.  \notag
\end{eqnarray}%
We now calculate $(\tau _{(n,1)})_{1}$ and $(\eta _{(n,0)})_{1}$. \ From
Lemma \ref{powermeasure} we know that 
\begin{equation*}
d\tau _{(n,1)}(t)=\frac{t^{\frac{1}{n}}}{y_{0}^{2}}d\tau (t^{\frac{1}{n}}),
\end{equation*}%
so that 
\begin{equation*}
d(\tau _{(n,1)})_{1}(t)=\frac{t}{y_{1}^{2}\cdots y_{n}^{2}}d\tau _{(n,1)}(t)=%
\frac{t^{1+\frac{1}{n}}}{y_{0}^{2}y_{1}^{2}\cdots y_{n}^{2}}d\tau (t^{\frac{1%
}{n}})\text{ (by Lemma \ref{restrictionBerger}).}
\end{equation*}%
On the other hand, and again using Lemma \ref{powermeasure}, we have 
\begin{equation*}
d\eta _{(n,0)}(t)=d\eta (t^{\frac{1}{n}}),
\end{equation*}%
so that 
\begin{equation*}
d(\eta _{(n,0)})_{1}(t)=\frac{t}{z_{0}^{2}\cdots z_{n-1}^{2}}d\eta
_{(n,0)}(t)=\frac{t}{z_{0}^{2}\cdots z_{n-1}^{2}}d\eta (t^{\frac{1}{n}})%
\text{ (again by Lemma \ref{restrictionBerger}).}\ 
\end{equation*}%
It follows from (\ref{eq1}) that 
\begin{eqnarray*}
d\psi ^{(0,1)}(t) &=&\frac{t^{1+\frac{1}{n}}}{y_{0}^{2}y_{1}^{2}\cdots
y_{n}^{2}}d\tau (t^{\frac{1}{n}})-\frac{a^{2}z_{0}^{2}\cdots z_{n-1}^{2}}{%
y_{1}^{2}\cdots y_{n}^{2}}\left\| \frac{1}{s}\right\| _{L^{1}(\xi )}\cdot 
\frac{t}{z_{0}^{2}\cdots z_{n-1}^{2}}d\eta (t^{\frac{1}{n}}) \\
&=&\frac{t}{y_{1}^{2}\cdots y_{n}^{2}}\{\frac{t^{\frac{1}{n}}}{y_{0}^{2}}%
d\tau (t^{\frac{1}{n}})-a^{2}\left\| \frac{1}{s}\right\| _{L^{1}(\xi )}d\eta
(t^{\frac{1}{n}})\} \\
&=&\frac{t}{y_{1}^{2}\cdots y_{n}^{2}}\{d\tau _{1}(t^{\frac{1}{n}%
})-a^{2}\left\| \frac{1}{s}\right\| _{L^{1}(\xi )}d\eta (t^{\frac{1}{n}})\}
\\
&=&\frac{t}{y_{1}^{2}\cdots y_{n}^{2}}d\psi (t^{\frac{1}{n}}).
\end{eqnarray*}%
It now readily follows that $\psi ^{(0,1)}\geq 0$ if and only if $\psi \geq
0 $, which establishes (i). \ 

To prove (ii), we begin by calculating $\psi ^{(0,0)}$. \ As in (\ref{eq1}),
we have 
\begin{eqnarray*}
d\psi ^{(0,0)}(t) &=&d(\tau ^{(0,0)})_{1}(t)-(a^{(0,0)})^{2}\left\| \frac{1}{%
s}\right\| _{L^{1}(\xi ^{(0,0)})}d\eta ^{(0,0)}(t) \\
&=&d(\tau _{(n,0)})_{1}(t)-\frac{a^{2}z_{0}^{2}\cdots z_{n-2}^{2}}{%
y_{1}^{2}\cdots y_{n-1}^{2}}\left\| \frac{1}{s}\right\| _{L^{1}(\xi )}d\eta
_{(n,n-1)}(t) \\
&=&\frac{t}{y_{0}^{2}\cdots y_{n-1}^{2}}d\tau (t^{\frac{1}{n}})-\frac{%
a^{2}z_{0}^{2}\cdots z_{n-2}^{2}}{y_{1}^{2}\cdots y_{n-1}^{2}}\left\| \frac{1%
}{s}\right\| _{L^{1}(\xi )}\frac{t^{\frac{n-1}{n}}}{z_{0}^{2}\cdots
z_{n-2}^{2}}d\eta (t^{\frac{1}{n}}) \\
&&\text{(by Lemmas \ref{restrictionBerger} and \ref{powermeasure})} \\
&=&\frac{1}{y_{0}^{2}y_{1}^{2}\cdots y_{n-1}^{2}}\{td\tau (t^{\frac{1}{n}%
})-a^{2}y_{0}^{2}\left\| \frac{1}{s}\right\| _{L^{1}(\xi )}t^{\frac{n-1}{n}%
}d\eta (t^{\frac{1}{n}})\}.
\end{eqnarray*}%
It follows that 
\begin{eqnarray*}
y_{0}^{2}y_{1}^{2}\cdots y_{n-1}^{2}\int \frac{1}{t}d\psi ^{(0,0)}(t)
&=&\int d\tau (t^{\frac{1}{n}})-a^{2}y_{0}^{2}\left\| \frac{1}{s}\right\|
_{L^{1}(\xi )}\int t^{-\frac{1}{n}}d\eta (t^{\frac{1}{n}}) \\
&=&1-a^{2}y_{0}^{2}\left\| \frac{1}{s}\right\| _{L^{1}(\xi )}\left\| \frac{1%
}{t}\right\| _{L^{1}(\eta )}.
\end{eqnarray*}%
On the other hand, 
\begin{eqnarray*}
y_{0}^{2}\int \frac{1}{t}d\psi (t) &=&y_{0}^{2}\{\int \frac{1}{t}d\tau
_{1}(t)-a^{2}\left\| \frac{1}{s}\right\| _{L^{1}(\xi )}\int \frac{1}{t}d\eta
(t)\} \\
&=&y_{0}^{2}\int \frac{1}{t}\cdot \frac{t}{y_{0}^{2}}d\tau
(t)-a^{2}y_{0}^{2}\left\| \frac{1}{s}\right\| _{L^{1}(\xi )}\left\| \frac{1}{%
t}\right\| _{L^{1}(\eta )} \\
&=&1-a^{2}y_{0}^{2}\left\| \frac{1}{s}\right\| _{L^{1}(\xi )}\left\| \frac{1%
}{t}\right\| _{L^{1}(\eta )}.
\end{eqnarray*}%
Thus, 
\begin{equation}
y_{0}^{2}y_{1}^{2}\cdots y_{n-1}^{2}\left\| \frac{1}{t}\right\| _{L^{1}(\psi
^{(0,0)})}=y_{0}^{2}\left\| \frac{1}{t}\right\| _{L^{1}(\psi )}.  \label{eq2}
\end{equation}%
Consider now 
\begin{equation*}
\varphi ^{(0,0)}=\sigma ^{(0,0)}-y_{0}^{2}\cdots
y_{n-1}^{2}\left\| \frac{1}{t}\right\| _{L^{1}(\psi ^{(0,0)})}\delta
_{0}-(a^{(0,0)})^{2}y_{0}^{2}\cdots y_{n-1}^{2}\left\| \frac{1}{t}\right\|
_{L^{1}(\eta ^{(0,0)})}\frac{\xi ^{(0,0)}}{s}.
\end{equation*}

We know that $\sigma
^{(0,0)}=\sigma $, that $a^{(0,0)}=\frac{az_{0}\cdots z_{n-2}}{y_{1}\cdots
y_{n-1}}$ and that $\xi ^{(0,0)}=\xi $, so using (\ref{eq2}) we obtain 
\begin{equation*}
\varphi ^{(0,0)}=\sigma -y_{0}^{2}\left\| \frac{1}{t}\right\| _{L^{1}(\psi
)}\delta _{0}-a^{2}y_{0}^{2}z_{0}^{2}\cdots z_{n-2}^{2}\left\| \frac{1}{t}%
\right\| _{L^{1}(\eta ^{(0,0)})}\frac{\xi }{s}.
\end{equation*}%
Since $\varphi =\sigma -y_{0}^{2}\left\| \frac{1}{t}\right\| _{L^{1}(\psi
)}\delta _{0}-a^{2}y_{0}^{2}\left\| \frac{1}{t}\right\| _{L^{1}(\eta )}\frac{%
\xi }{s}$, it is easy to see that it suffices to prove that $z_{0}^{2}\cdots
z_{n-2}^{2}\left\| \frac{1}{t}\right\| _{L^{1}(\eta ^{(0,0)})}=\left\| \frac{%
1}{t}\right\| _{L^{1}(\eta )}$. \ We know that $\eta ^{(0,0)}=\eta
_{(n,n-1)} $, so 
\begin{eqnarray*}
z_{0}^{2}\cdots z_{n-2}^{2}\left\| \frac{1}{t}\right\| _{L^{1}(\eta
^{(0,0)})} &=&z_{0}^{2}\cdots z_{n-2}^{2}\int \frac{1}{t}d\eta _{(n,n-1)}(t)
\\
&=&z_{0}^{2}\cdots z_{n-2}^{2}\int \frac{1}{t}\frac{t^{\frac{n-1}{n}}}{%
z_{0}^{2}\cdots z_{n-2}^{2}}d\eta (t^{\frac{1}{n}}) \\
&=&\int t^{-\frac{1}{n}}d\eta (t^{\frac{1}{n}})=\left\| \frac{1}{t}\right\|
_{L^{1}(\eta )},
\end{eqnarray*}%
as desired. \qed
\end{proof}

\begin{corollary}
\label{cor11}Let $\left\langle \sigma ,\tau ,a,\xi ,\eta \right\rangle \in 
\mathcal{TC}$, and let $n\geq 2$. \ Assume that \newline
$\left\langle \sigma ,\tau ,a,\xi ,\eta \right\rangle ^{(1,n)}$ is
subnormal. \ Then $\left\langle \sigma ,\tau ,a,\xi ,\eta \right\rangle $ is
subnormal.
\end{corollary}

\begin{proof}
Assume that $\left\langle \sigma ,\tau ,a,\xi ,\eta \right\rangle ^{(1,n)}$
is subnormal, and recall that the power of a $2$-variable weighted shift
splits as an orthogonal direct sum of $2$-variable weighted shifts. \
Moreover, each summand is in $\mathcal{TC}$ (because $\left\langle \sigma
,\tau ,a,\xi ,\eta \right\rangle \in \mathcal{TC}$). \ The fact that $%
\left\langle \sigma ,\tau ,a,\xi ,\eta \right\rangle ^{(1,n)}$ is subnormal
readily implies that each direct summand is subnormal, and then Lemma \ref%
{LemmaROMP} says that $\psi ^{(0,q)}\geq 0$ and $\varphi ^{(0,q)}\geq 0\;$%
(all $q\geq 0$). \ In particular, $\psi ^{(0,1)}\geq 0$ and $\varphi
^{(0,0)}\geq 0$. \ It follows from Proposition \ref{connection} that $\psi
\geq 0$ and $\varphi \geq 0$. \ Applying Lemma \ref{LemmaROMP} once again,
we see that $\left\langle \sigma ,\tau ,a,\xi ,\eta \right\rangle $ is
subnormal. \qed
\end{proof}

\begin{corollary}
\label{cor12}Let $\left\langle \sigma ,\tau ,a,\xi ,\eta \right\rangle \in 
\mathcal{TC}$, and let $m\geq 2$. \ Assume that \newline
$\left\langle \sigma ,\tau ,a,\xi ,\eta \right\rangle ^{(m,1)}$ is
subnormal. \ Then $\left\langle \sigma ,\tau ,a,\xi ,\eta \right\rangle $ is
subnormal.
\end{corollary}

\section{Proof of the main theorem}

We are now ready to prove our main result, which we restate for the reader's
convenience.

\begin{theorem}
\label{Thm1}Let $\mathbf{T}\in \mathcal{TC}$. \ The
following statements are equivalent.\newline
(i) $\mathbf{T}\in \mathfrak{H}_{\infty }$;\newline
(ii) $\mathbf{T}^{(m,n)}$ $\in \bigoplus \mathfrak{H%
}_{\infty }$ for all $m,n\geq 1$;\newline
(iii) $\mathbf{T}^{(m,n)}\in \bigoplus \mathfrak{H}%
_{\infty }$ for some $m,n\geq 1$.
\end{theorem}

\begin{proof}
It is clear that $(i)\Longrightarrow (ii)$ and that $(ii)\Longrightarrow
(iii)$. \ Assume that $(iii)$ holds, with $n\geq 2$. \ Since $\mathbf{T}^{(m,n)}=(\mathbf{T}^{(m,1)})^{(1,n)}$, we can use
Corollary \ref{cor11} to conclude that $\mathbf{T}^{(m,1)}$ is
subnormal. \ If we now apply Corollary \ref{cor12}, we obtain that $\mathbf{T}$ is subnormal, as desired. \qed
\end{proof}

\section{\label{Example} An application}

In our previous work (\cite{CuYo1}, \cite{CuYo2}, \cite{CuYo3}, \cite{CuYo5}, \cite{CLY1}, \cite{CLY2}, \cite{ROMP}, \cite{CLY4}, \cite{Yo1}, \cite{Yo2}), we have shown that there are many different families of commuting pairs
of subnormal operators, jointly hyponormal but not admitting commuting
normal extensions, that is, $\mathbf{T}\in \mathfrak{H%
}_{1}$ but $\mathbf{T}\notin \mathfrak{H}_{\infty }$
(all $m,n\geq 1$). \ As a simple application of Theorem \ref{Thm1}, we now
show that $\mathfrak{H}_{1}\cap \mathcal{TC}\neq \mathfrak{H}_{\infty }\cap 
\mathcal{TC}$; moreover, there exists $\mathbf{T}\in 
\mathcal{TC}$, such that $\mathbf{T}\in \mathfrak{H}%
_{1}$ but $\mathbf{T}^{(m,n)}\notin \bigoplus 
\mathfrak{H}_{\infty }$ (all $m,n\geq 1$). \ We recall that $\operatorname{shift}%
(x_{0},x_{1},\cdots )$ and $\operatorname{shift}(y_{0},y_{1},\cdots )$ are subnormal
unilateral weighted shifts with Berger measures $\sigma $ and $\tau $,
respectively. \ Consider a contractive $2$-variable weighted shift $\mathbf{T}\in \mathfrak{H}_{0}$ whose weight diagram
is given by Figure \ref{figureex}(i); that is, in the $5$-tuple $\left\langle
\sigma ,\tau ,a,\xi ,\eta \right\rangle $, we have

\vspace{6pt} 
$%
\begin{array}[t]{ll}
\ast & d\sigma (t):=(1-\kappa ^{2})d\delta _{0}(t)+\frac{\kappa ^{2}}{2}dt+\frac{%
\kappa ^{2}}{2}d\delta _{1}(t), \\ 
\ast & 
\begin{array}{c}
\tau \text{ is the Berger measure of shift }(y_{0},y_{1},\cdots ),\text{with 
}\tau _{1}\text{ the }2\text{-atomic Berger} \\ 
\multicolumn{1}{l}{
\text{measure of the Stampfli subnormal completion of }\sqrt{\omega _{0}}<\sqrt{%
\omega _{1}}<\sqrt{\omega _{2}},}
\end{array}
\\ 
\ast & a\text{ is a positive number,} \\ 
\ast & \xi :=\delta _{1},\text{ and } \\ 
\ast & \eta :=\delta _{1}.%
\end{array}%
$

\begin{example}
\label{Main 2} Let $\mathbf{T} \equiv \left\langle
\sigma ,\tau ,a,\xi ,\eta \right\rangle $ be the $2$-variable
weighted shift given by Figure \ref{figureex}(i), with $\sigma $, $\tau
_{1}$, $a$, $\xi$ and $\eta$ as above. \ Then $\mathbf{T}\in \mathfrak{H}%
_{1} $ and $\mathbf{T}^{(m,n)}\notin \mathfrak{H}%
_{\infty } $ (all $m,n\geq 1$) if and only if $s\left( \kappa \right)
<y_{0}<h\left( \kappa \right) $, where%
\begin{equation*}
\begin{tabular}{l}
$s\left( \kappa \right) :=\min \left\{ \frac{\sqrt{t_{1}}}{a}\sqrt{\rho _{1}}%
,\sqrt{\frac{(1-\kappa ^{2})}{\left\| \frac{1}{t}\right\| _{L^{1}\left( \tau
_{1}\right) }-\frac{a^{2}}{t_{1}}}},\frac{\sqrt{t_{1}}}{a}\sqrt{\frac{\kappa
^{2}}{2}},\sqrt{\frac{1}{\left\| \frac{1}{t}\right\| _{L^{1}\left( \tau
_{1}\right) }}}\right\} $%
\end{tabular}%
\end{equation*}%
and%
\begin{equation*}
\begin{tabular}{l}
$h\left( \kappa \right) :=\sqrt{\frac{x_{0}^{2}y_{1}^{2}(x_{1}^{2}-x_{0}^{2})%
}{x_{0}^{2}(x_{1}^{2}-x_{0}^{2})+(a^{2}-x_{0}^{2})^{2}}}.$%
\end{tabular}%
\end{equation*}
\end{example}

\setlength{\unitlength}{.88mm} \psset{unit=.88mm} 
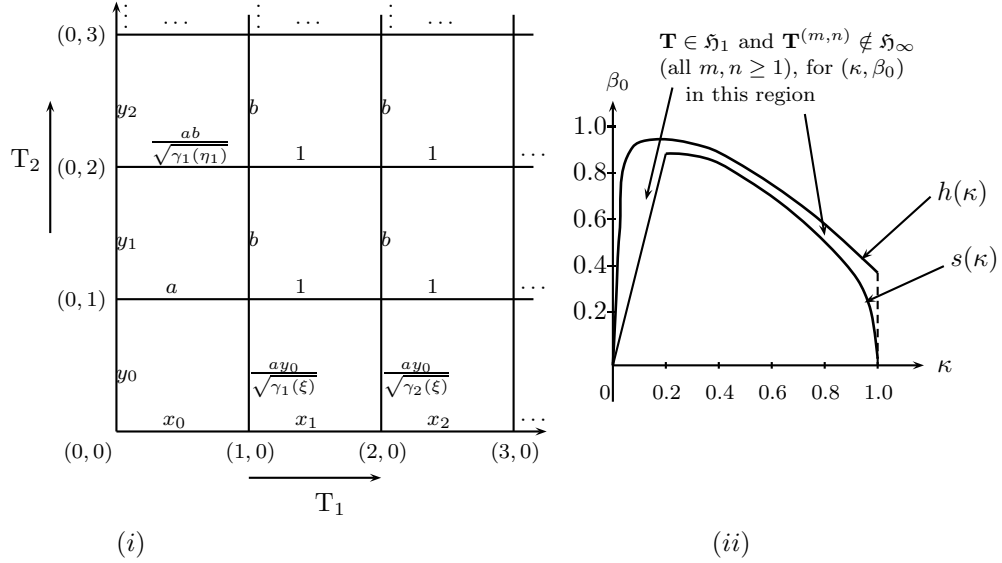
\begin{figure}[th]
\begin{center}
\begin{picture}(160,90)
\psline{->}(20,20)(85,20)
\psline(20,40)(83,40)
\psline(20,60)(83,60)
\psline(20,80)(83,80)
\psline{->}(20,20)(20,85)
\psline(40,20)(40,83)
\psline(60,20)(60,83)
\psline(80,20)(80,83)

\put(12,16){\footnotesize{$(0,0)$}}
\put(36.5,16){\footnotesize{$(1,0)$}}
\put(56.5,16){\footnotesize{$(2,0)$}}
\put(76.5,16){\footnotesize{$(3,0)$}}

\put(27,21){\footnotesize{$x_{0}$}}
\put(47,21){\footnotesize{$x_{1}$}}
\put(67,21){\footnotesize{$x_{2}$}}
\put(81,21){\footnotesize{$\cdots$}}

\put(27.5,41){\footnotesize{$a$}}
\put(47,41){\footnotesize{$1$}}
\put(67,41){\footnotesize{$1$}}
\put(81,41){\footnotesize{$\cdots$}}

\put(25,63){\footnotesize{$\frac{ab}{\sqrt{\gamma_{1}(\eta_{1})}}$}}
\put(47,61){\footnotesize{$1$}}
\put(67,61){\footnotesize{$1$}}
\put(81,61){\footnotesize{$\cdots$}}

\put(27,81){\footnotesize{$\cdots$}}
\put(47,81){\footnotesize{$\cdots$}}
\put(67,81){\footnotesize{$\cdots$}}

\psline{->}(40,13)(60,13)
\put(50,8){$\rm{T}_1$}
\psline{->}(10,50)(10,70)
\put(4,60){$\rm{T}_2$}

\put(11,39){\footnotesize{$(0,1)$}}
\put(11,59){\footnotesize{$(0,2)$}}
\put(11,79){\footnotesize{$(0,3)$}}

\put(20,28){\footnotesize{$y_{0}$}}
\put(20,48){\footnotesize{$y_{1}$}}
\put(20,68){\footnotesize{$y_{2}$}}
\put(21,81){\footnotesize{$\vdots$}}

\put(40,28){\footnotesize{$\frac{ay_{0}}{\sqrt{\gamma_{1}(\xi)}}$}}
\put(40,48){\footnotesize{$b$}}
\put(40,68){\footnotesize{$b$}}
\put(41,81){\footnotesize{$\vdots$}}

\put(60,28){\footnotesize{$\frac{ay_{0}}{\sqrt{\gamma_{2}(\xi)}}$}}
\put(60,48){\footnotesize{$b$}}
\put(60,68){\footnotesize{$b$}}
\put(61,81){\footnotesize{$\vdots$}}

\put(20,2){$(i)$}


\put(110,2){$(ii)$}

\psline{->}(95,30)(95,70)
\psline{->}(95,30)(142,30)
\put(94,38){$\_$}
\put(94,45){$\_$}
\put(94,52){$\_$}
\put(94,59){$\_$}
\put(94,66){$\_$}
\put(89,37){$0.2$}
\put(89,44){$0.4$}
\put(89,51){$0.6$}
\put(89,58){$0.8$}
\put(89,65){$1.0$}
\put(144,29){$\kappa$}
\psline{-}(95,24.5)(95,30.5)
\psline{-}(103,29.5)(103,30.5)
\psline{-}(111,29.5)(111,30.5)
\psline{-}(119,29.5)(119,30.5)

\psline{-}(127,29.5)(127,30.5)
\psline{-}(135,29.5)(135,30.5)

\put(93,25){\footnotesize{$0$}}
\put(101,25){\footnotesize{$0.2$}}
\put(109,25){\footnotesize{$0.4$}}
\put(117,25){\footnotesize{$0.6$}}
\put(125,25){\footnotesize{$0.8$}}
\put(133,25){\footnotesize{$1.0$}}
\put(94,72){\footnotesize{$\beta_{0}$}}

\put(102,78){\footnotesize{$\mathbf{T}\in\mathfrak{H}_{1}$ and
$\mathbf{T}^{(m,n)}\notin \mathfrak{H}_{\infty}$}}

\put(102,74){\footnotesize{(all $m,n\geq 1$), for $(\kappa,\beta_{0})$}}
\put(106,70){\footnotesize{in this region}}
\psline{->}(104,72)(100,55)
\psline{->}(124,68)(127,49.6)

\psline{-}(95,30)(103,62)
\pscurve[linewidth=1pt](103,62)(105,62)(110,61)(115,58.2)(120,54.8)(125,50.6)(130,45.5)
(132,42.9)(133,41)(134,38.1)(135,31)

\pscurve[linewidth=1pt](95,30)(95.5,40)(96,50)(98,63)(100,64)(109,63)(113,61.1)(125,53)(135,44)
\psline[linestyle=dashed,dash=3pt 2pt]{-}(135,30)(135,44)

\psline{->}(143,55.5)(132.5,46) \put(144,55){$h(\kappa)$}
\psline{->}(145,45)(133,39.5) \put(146,45){$s(\kappa)$}

\end{picture}
\end{center}
\caption{Weight diagram of the 2-variable weighted shift in Example \ref%
{Main 2} and graphs of $s(\protect\kappa )$ and $h(\protect\kappa )$ on the
interval $[0,1]$, respectively.}
\label{figureex}
\end{figure}

Figure \ref{figureex}(ii) specifies a region in the $(\kappa,\beta_0)$ plane where $\mathbf{T}$ is
hyponormal but none of its powers is subnormal. \ A detailed analysis of Example \ref{Main 2} and of other applications of Theorem \ref{Main 1} will be discussed elsewhere.

\medskip \textit{Acknowledgments}. \ The authors are deeply indebted to the referee for several suggestions that improved the presentation. \ The specific formulas for $s(\kappa)$ and $h(\kappa)$ in Example 8.1 were obtained using calculations with the software tool \textit{Mathematica} \cite{Wol}.


\begin{thebibliography}{99}
\bibitem{Abr} M. Abrahamse, Commuting subnormal operators, \textit{Ill.
Math. J.}. 22(1978), 171-176.

\bibitem{Ath} A. Athavale, On joint hyponormality of operators, \textit{\
Proc. Amer. Math. Soc}. 103(1988), 417-423.


\bibitem{Con} J. Conway, \textit{The Theory of Subnormal Operators,}
Mathematical Surveys and Monographs, vol. 36, Amer. Math. Soc., Providence,
1991.

\bibitem{bridge} R. Curto, Joint hyponormality: A bridge between
hyponormality and subnormality, \textit{Proc. Symposia Pure Math.} 51(1990),
69-91.

\bibitem{QHWS} R. Curto, Quadratically hyponormal weighted shifts, \textit{%
Integral Equations Operator Theory} 13(1990), 49-66.

\bibitem{CLY1} R. Curto, S.H. Lee and J. Yoon, $k$-hyponormality of
multivariable weighted shifts, \textit{J. Funct. Anal. } 229(2005), 462-480.

\bibitem{CLY2} R. Curto, S.H. Lee and J. Yoon, Hyponormality and
subnormality for powers of commuting pairs of subnormal operators, \textit{%
J. Funct. Anal. }245(2007), 390-412.

\bibitem{ROMP} R. Curto, S.H. Lee and J. Yoon, Reconstruction of the Berger
measure when the core is of tensor form, \textit{Actas del XVI Coloquio
Latinoamericano de \'{A}lgebra, Bibl. Rev. Mat. Iberoamericana} (2007),
317-331.

\bibitem{CLY4} R. Curto, S.H. Lee and J. Yoon, Which $2$-hyponormal $2$%
-variable weighted shifts are subnormal?, \textit{Linear Algebra Appl.}
429(2008), 2227-2238.

\bibitem{CMX} R. Curto, P. Muhly and J. Xia, Hyponormal pairs of commuting
operators, \textit{Operator Theory: Adv. Appl.} 35(1988), 1-22.

\bibitem{CuP} R. Curto and S. Park, $k$-hyponormality of powers of weighted
shifts, \textit{Proc. Amer. Math. Soc}. 131(2003), 2761-2769.

\bibitem{CP1} R. Curto and M. Putinar, Existence of non-subnormal
polynomially hyponormal operators, \textit{Bull. Amer. Math. Soc.} (N.S.)
25(1991) 373-378.

\bibitem{CP2} R. Curto and M. Putinar, Nearly subnormal operators and
moments problems, \textit{J. Funct. Anal. 115}(1993), 480-497.

\bibitem{CuYo1} R. Curto and J. Yoon, Jointly hyponormal pairs of subnormal
operators need not be jointly subnormal, \textit{Trans. Amer. Math. Soc.}
358(2006), 5139-5159.

\bibitem{CuYo2} R. Curto and J. Yoon, Disintegration-of-measure techniques
for multivariable weighted shifts, \textit{Proc. London Math. Soc.}
93(2006), 381-402.

\bibitem{CuYo3} R. Curto and J. Yoon, Propagation phenomena for hyponormal $%
2 $-variable weighted shifts, \textit{J. Operator Theory} 58(2007), 175-203.


\bibitem{CuYo5} R. Curto and J. Yoon, When is hyponormality for $2$-variable
weighted shifts invariant under powers?, Indiana Univ. Math. J., to appear.

\bibitem{GeWa} R. Gellar and L.J. Wallen, Subnormal weighted shifts and the
Halmos-Bram criterion,\textit{\ Proc. Japan Acad.}, 46(1970), 375-378.

\bibitem{JeLu} N.P. Jewell and A.R. Lubin, Commuting weighted shifts and
analytic function theory in several variables, \textit{J. Operator Theory}
1(1979), 207-223.

\bibitem{Lu1} A. Lubin, Weighted shifts and product of subnormal operators, 
\textit{Indiana Univ. Math. J.} 26(1977), 839-845.

\bibitem{Lu2} A. Lubin, Extensions of commuting subnormal operators, \textit{%
Lecture Notes in Math.} 693(1978), 115-120.

\bibitem{Lu3} A. Lubin, A subnormal semigroup without normal extension, 
\textit{Proc. Amer. Math. Soc.} 68(1978), 176-178.

\bibitem{McCP} S. McCullough and V. Paulsen, A note on joint hyponormality, 
\textit{Proc. Amer. Math. Soc. 107}(1989), 187-195.


\bibitem{Sta} J. Stampfli, Which weighted shifts are subnormal?, \textit{%
Pacific J. Math}. 17(1966), 367-379.

\bibitem{Wol} Wolfram Research, Inc. \textit{Mathematica}, Version 4.2, 
\textit{Wolfram Research Inc.}, Champaign, IL, 2002.

\bibitem{Yo1} J. Yoon, Disintegration of measures and contractive $2$%
-variable weighted shifts, \textit{Integral Equations Operator Theory},
59(2007), 281-298.

\bibitem{Yo2} J. Yoon, Schur product techniques for commuting multivariable
weighted shifts, \textit{J. Math. Anal. Appl.}, 333(2007), 626-641.
\end{thebibliography}
\end{document}